\documentclass[a4paper,11pt, oneside]{amsart}
\usepackage{enumerate}
\usepackage{amsfonts,amssymb}
\usepackage{float}
\usepackage{graphicx,tikz}
\usepackage{mathtools}
\usepackage{tikz-cd} 
\usepackage{float}
\usepackage{amsmath, amsthm, amsfonts}
\usepackage{geometry}
\usepackage{hyperref}
\usepackage{enumitem}  
\usepackage[mathscr]{euscript}
\usepackage[capitalise]{cleveref}
\usepackage{comment}
\usepackage{enumitem}
\usepackage{amsrefs}
\usepackage[official]{eurosym}
\usepackage{mathtools}
\usepackage{graphicx}
\usepackage{nomencl}
\usepackage{amsfonts}
\usepackage{hyperref}
\usepackage{amssymb}
\usepackage{fancyhdr}
\usepackage{amscd}
\usepackage[english]{babel}
\usepackage{easybmat}
\usepackage{multirow,bigdelim}

\theoremstyle{plain}
\theoremstyle{definition}

\newtheorem*{theorem*}{Theorem}
\newtheorem{thm}{Theorem}
\newtheorem{pr}[thm]{Proposition}

\newtheorem{lem}[thm]{Lemma}

\newtheorem*{theo}{Theorem}
\theoremstyle{definition}
\newtheorem{dfn}[thm]{\scshape{Definition}}

\newtheorem{rem}[thm]{\scshape{Remark}}

\newcommand\T{{\mathcal{T}}}
\newcommand\LL{{\mathcal{L}}}
\makeatletter
\def\cleardoublepage{\clearpage\if@twoside \ifodd\c@page\else
	\hbox{}
	\thispagestyle{empty}
	\newpage
	\if@twocolumn\hbox{}\newpage\fi\fi\fi}
\makeatother
\DeclareMathOperator{\Aut}{Aut}
\DeclareMathOperator{\St}{st}
\DeclareMathOperator{\Ker}{Ker}

\def\LL{\mathcal{L}}

\def\M{\mathscr{M}}

\keywords{Groups of automorphisms of rooted trees, branch groups}
\subjclass[2010]{20E08}

\begin{document}
	
	\title[A family of fractal non-contracting weakly branch groups]{A family of fractal non-contracting weakly branch groups}
\author[M. Noce]{Marialaura Noce}
\address{University of Bath, Department of Mathematical Sciences, Claverton Down, Bath, BA2 7AY, U.K.}
\email{mn670@bath.ac.uk}

\thanks{The author is supported by EPSRC (grant number
1652316), and partially by the Spanish Government, grant MTM2017-86802-P, partly with FEDER funds.}

    \begin{abstract}
We construct a new family of groups that is non-contracting and weakly regular branch over the derived subgroup. This gives the first example of an infinite family of groups acting on a $d$-adic tree, with $d \geq 2$, with these properties. 
	\end{abstract}

	\maketitle
	

\section{Introduction}

Weakly branch groups were first defined by Grigorchuk in 1997 as a generalization of the famous $p$-groups constructed by Grigorchuk himself \cites{Grig1,Grig2}, and Gupta and Sidki \cite{GS}. These groups possess remarkable and exotic properties. For instance, the Grigorchuk group is the first example of a group of intermediate word growth, and amenable but not elementary amenable. Also, together with the Grigorchuk group, other subgroups of the group of automorphisms of rooted trees like the Gupta-Sidki $p$-groups and many groups in the family of the so-called Grigorchuk-Gupta-Sidki groups have been shown to be a counterexample to the General Burnside Problem.

For these reasons, (weakly) branch groups spread great interest among group
theorists, who have actively investigated further properties of these in the recent years: just-infiniteness, fractalness, maximal subgroups, or contraction. 

Roughly speaking, a group is said to be \textit{contracting} if the sections of every element are ``shorter'' than the element itself, provided the element does not belong to a fixed finite set, called the nucleus (see the exact definition in Section 2).

Even though in the literature there are many examples of weakly branch contracting groups, not much is known about weakly branch groups that are non-contracting. In 2005 Dahmani \cite{dahmani} provided the first example of a non-contracting weakly regular branch automaton group. Another example with similar properties was constructed by Mamaghani in 2011 \cite{mohamed}. Both are examples of groups acting on the binary tree.

In this paper we construct the first example of an infinite family of non-contracting weakly branch groups acting on $d$-adic trees for any $d \geq 2$. This result gives wealth of examples of groups with these unusual properties. In the following we denote with $\Aut\T_d$ the group of automorphisms of a $d$-adic tree.

\begin{theo}
For any $d \geq 2$, there exists a group $\mathscr{M}(d) \leq \Aut\T_d$ that is weakly regular branch over its derived subgroup, non-contracting and fractal.
\end{theo}


\subsubsection*{Organization}
In Section 2 we give some definitions of groups acting on regular rooted trees and of properties like fractalness, branchness and contraction. In Section 3 we introduce these groups and we prove the main theorem together with some additional results regarding the order of elements of $\M(d)$.

\vspace{0.4cm}
\subsection*{Acknowledgements}
The author wants to thank Gustavo A. Fern\'andez-Alcober and Albert Garreta for useful discussions.

\vspace{0.6cm}
\section{Preliminaries}
In this section we fix some terminology regarding groups of automorphisms of $d$-adic (rooted) trees. For further information on the topic, see \cite{branch} or \cite{nekrashevych}.

Let $d$ be a positive integer, and $\T_d$ the $d$-adic tree. We denote with $\Aut\T_d$ the group of automorphisms of $\T_d$. We let $\LL_n$ be the $n$-th level of $\T_d$, and $\LL_{\geq n}$ the levels of the tree from level $n$ and below.

The stabilizer of a vertex $u$ of the tree is denoted by $\St(u)$, and, more generally, the $n$-th \emph{level stabilizer} $\St(n)$ is the subgroup of $\Aut\T_d$ that fixes every vertex of $\LL_n$. If $G \leq \Aut\T_d$, we define the $n$-th level stabilizer of $G$ as $\St_{G}(n) = \St(n) \cap G$.
Notice that stabilizers are normal subgroups of the corresponding group. We let $\psi$ be the isomorphism
\begin{align*}
\psi: \St(1) & \longrightarrow \Aut\T_d \times \overset{d}{\cdots} \times \Aut\T_d \label{psi_iso}\\
g & \longmapsto (g_{u})_{u \in \LL_1}, \nonumber
\end{align*}
where $g_{u}$ is the section of $g$ at the vertex $u$, i.e.\ the action of $g$ on the subtree $\T_u$ that hangs from the vertex $u$. Let $S_d$ be the symmetric group on $d$ letters. An automorphism $a \in \Aut\T_d$ is called \emph{rooted} if there exists a permutation $\sigma \in S_{d}$ such that $a$ rigidly permutes the trees $\{\T_u \mid u\in \LL_1\}$  according to the permutation $\sigma$. We usually identify $a$ and $\sigma$.

Notice that if $g \in \St(1)$ with $\psi(g)=(g_1, \dots, g_d)$, and $\sigma$ is a rooted automorphism, then,
\begin{equation}\label{carmine}
\psi(g^{\sigma})= \left(g_{\sigma^{-1}(1)},  \dots, g_{\sigma^{-1}(d)}\right).
\end{equation}
Any element $g \in G$ can be written uniquely in the form $g= h \sigma$, where $h \in \St(1)$ and $\sigma$ is a rooted automorphism.

Notice also that the decomposition $g=h \sigma$, together with the action \eqref{carmine}, yields  isomorphisms
\begin{align}
\begin{split}\label{e: iterated_wreath}
\Aut\T_d &\cong \St(1) \rtimes S_d \cong \left(\Aut\T_d \times \overset{d}{\dots} \times \Aut\T_d\right) \rtimes S_d \\ 
&\cong \Aut\T_d \wr S_d \cong (( \dots \wr S_d) \wr S_d) \wr S_d. 
\end{split}
\end{align}

Throughout the paper, we will use the following shorthand notation: let $f \in \Aut\T$ of the form $f=gh$, where $g\ \in \St_G(1)$ and $h$ is the rooted automorphism corresponding to the permutation $\sigma \in S_d$. If $\psi(g) = (g_1, \dots, g_d)$,  we write $f = (g_1, \dots, g_d)\sigma$.

\begin{dfn}
Let $G\leq \Aut\T_d$, and let $V(\T_d)$ be the set of vertices of $\T_d$. Then:
\begin{itemize}

\item The group $G$ is said to be \emph{self-similar} if for any $g \in G$ we have
$$
\{g_u \mid g \in G, u \in V(\T_d)\} \subseteq G.
$$
In other words, the sections of $g$ at any vertex are still elements of $G$. For example, $\Aut\T_d$ is self-similar. 
\item A self-similar group $G$ is said to be \emph{fractal} if $\psi_u(\St_G(u))=G$ for all $u \in V(\T_d)$, where $\psi_u$ is the homomorphism sending $g \in \St(u)$ to its section $g_u$. 
\end{itemize}
\end{dfn}

To prove that a group is self-similar it suffices to show that the condition above is satisfied by the vertices of the first level of the tree (see \cite[Proposition 3.1]{hanoi}). The situation is similar in the case of fractal groups. More precisely, using Lemma \ref{jongrig}, we deduce that to show that a group $G$ is fractal, it is enough to check the vertices in the first level of $\T_d$.

We recall that $G$ is said to be \emph{level transitive} if it acts transitively on every level of the tree.
\begin{lem}\cite[Lemma 2.7]{Jone}\label{jongrig}
If $G \leq \Aut\T_d$ is transitive on the first level and $\psi_x(\St_G(x)) = G$ for some $x \in \LL_1$, then $G$ is fractal and level transitive.
\end{lem}

Here we present a family of non-contracting weakly branch groups. To this end, in the following, we recall the corresponding two definitions.

\begin{dfn}
A self-similar group $G \leq \Aut\T_d$  is \textit{contracting} if there exists a finite subset $\mathcal{F} \subseteq G$ such that for every $g \in G$ there is $n$ such that $g_v$ belongs to $\mathcal{F}$ for all vertices $v$ of $\LL_{\geq n}$. The smallest set among all these finite sets is called the \emph{nucleus} of $G$ and it is denoted by $\mathcal{N}$.
\end{dfn}

\begin{dfn}
Let $G$ be a self-similar subgroup of $\Aut\T_d$.
We say that $G$ is \textit{weakly regular branch} over a subgroup $K \leq G$ if $G$ is level transitive and we have
$$
\psi(K \cap \St_G(1)) \geq K \times \dots \times K.
$$
If, additionally, $K$ is of finite index in $G$, then $G$ is said to be \textit{regular branch} over $K$.
\end{dfn}

\vspace{0.6cm}
\section{The groups $\M(d)$}

Let $d \geq 2$, and let $\T_d$ be the $d$-adic tree. The group $\M(d) \leq \Aut\T_d$ is generated by $d$ elements $m_1, \dots, m_d$, where $m_1, \dots, m_d$ are defined recursively as follows:
\begin{align*}
m_1 & = (1, \dots, 1, m_1)(1 \ \dots \ d) \\
m_2 & = (1, \dots, 1, m_2, 1)(1 \ \dots \ d-1) \\
m_3 & = (1, \dots, m_3, 1, 1)(1 \ \dots \ d-2) \\
\vdots \\
m_{d-1} & = (1,m_{d-1}, 1, \dots, 1)(1 \ 2) \\
m_{d} & = (m_1, \dots, m_d).
\end{align*}

For example, for $d=3$, we have $\M(3)=\langle m_1, m_2, m_3 \rangle$, where 
$$
m_1 = (1,1, m_1)(1 \ 2 \ 3), \quad m_2 = (1, m_2, 1)(1 \ 2), \quad m_3=(m_1,m_2,m_3).
$$

\begin{figure}[H]
 \centering
\includegraphics[scale=0.7]{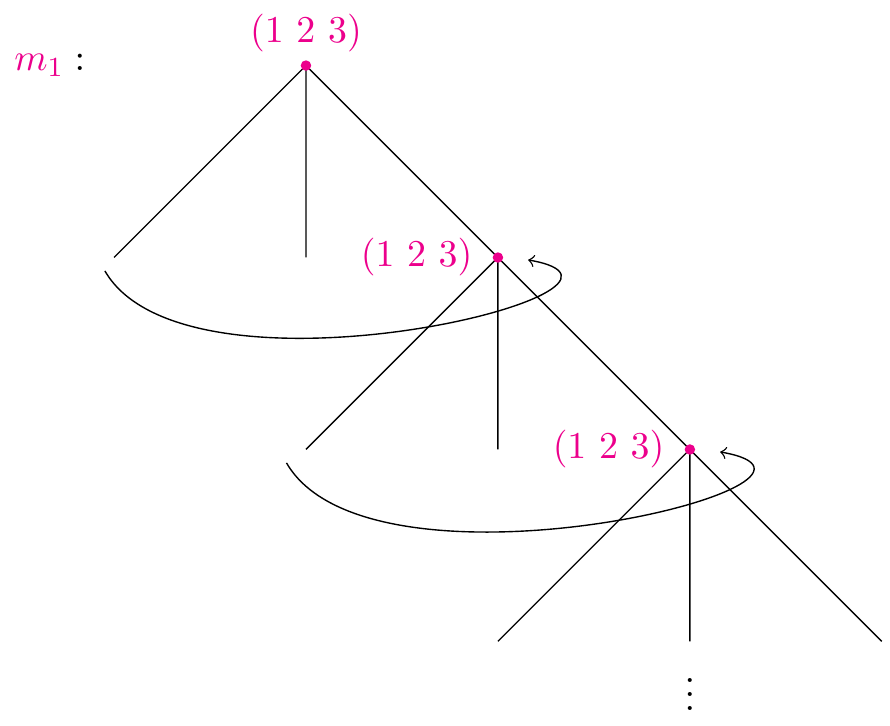}
\includegraphics[scale=0.7]{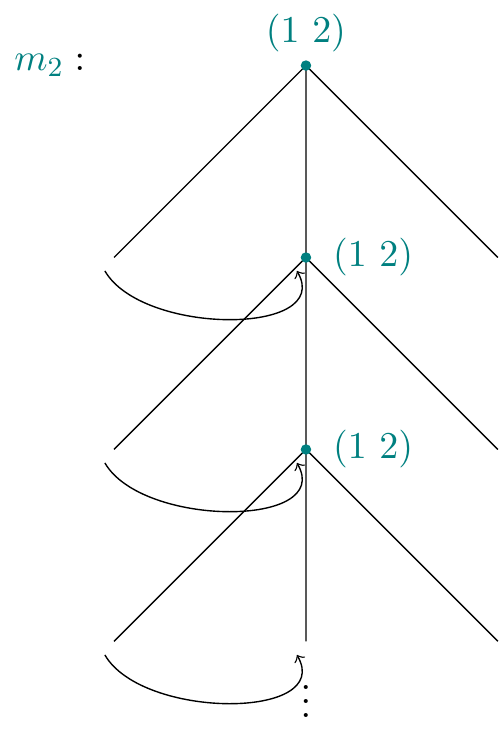}
\includegraphics[scale=0.6]{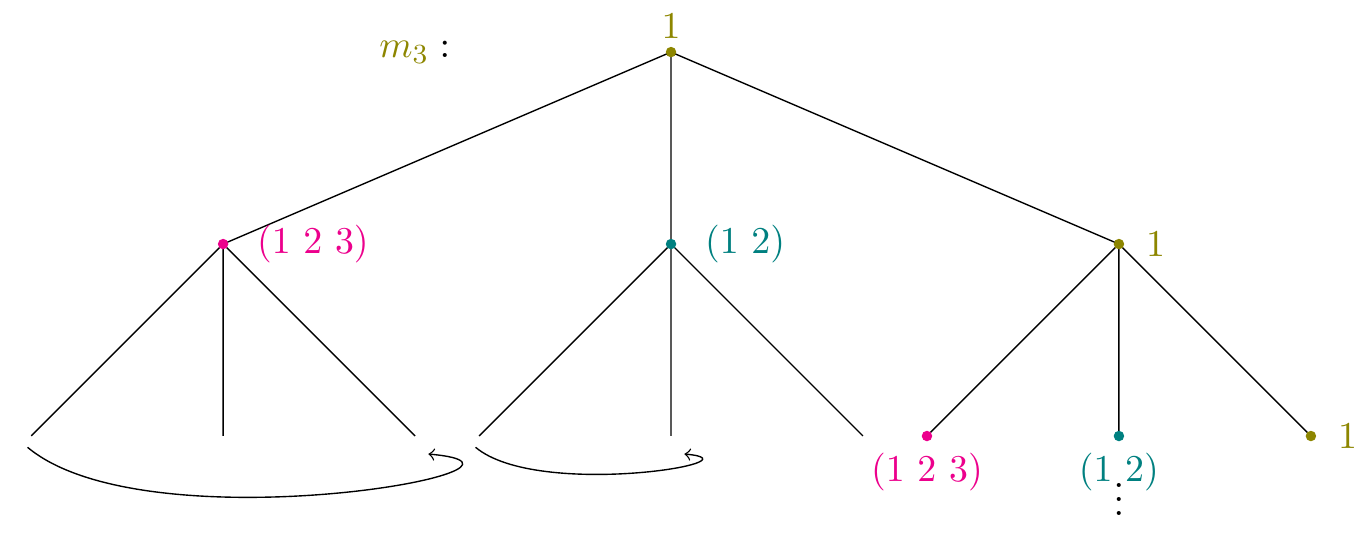}
    \caption{The generators of $\M(3)$.}\nonumber
    \label{m3}
\end{figure}

\subsection{Proof of the main theorem}

In this section we prove the main result of the paper. 

In order to ease notation, and unless it is strictly necessary, we will simply write $\M$ to denote an arbitrary group $\M(d)$. 

\begin{pr}
The group $\M$ is fractal and level transitive.
\end{pr}
\begin{proof}
Notice that the group is transitive on the first level because the rooted part of the generator $m_1$ is $(1 \ 2 \ \dots \ d)$. Also, it is straightforward to see that the group is self-similar, since the sections of every generator at the first level are generators of $\M$. To see that $\M$ is fractal, note that
\begin{align*}
m_1^d & = (m_1, \dots, m_1) \\
m_d^{m_1^{d-2}} & = (m_{3}^{m_1}, \dots, m_2) \\
\vdots \\
m_d^{m_1^2} & = (m_{d-1}^{m_1}, \dots, m_{d-2}) \\
m_d^{m_1} & = (m_d^{m_1}, \dots, m_{d-1}) \\
m_d & =(m_1, \dots, m_d).
\end{align*}
Then in the last component of the elements above we obtain all the generators of $\M$. Using Lemma \ref{jongrig}, we conclude that $\M$ is level transitive and fractal.
\end{proof}

\begin{pr}\label{mweaklybranch}
The group $\M$ is weakly regular branch over its derived subgroup $\M'$. 
\end{pr}
\begin{proof}
We will distinguish the case $d=2$, and $d \geq 3$ separately.
Let $d=2$. The element $[m_1,m_2]$ is non-trivial since
$$
[m_1,m_2]=(m_1^{-1}m_2^{-1}m_1^2,m_1^{-1}m_2),
$$
and $m_1^{-1}m_2 \notin \St_{\M}(1)$. Then $\M(2)'$ is non-trivial, and we have 
\begin{equation}\label{m1m2}
[m_1^2,m_2]=(1, [m_1,m_2]).
\end{equation}
From Equation \eqref{m1m2} and since $\M(2)'=\langle[m_1,m_2]\rangle^{\M(2)}$, we obtain that $\{1\}\times\M(2)'\leq \psi(\M(2)')$. As $\M(2)$ is level transitive, we conclude that $\M(2)' \times\M(2)'\leq \psi(\M(2)')$, as desired.

Let $d \geq 3$, and write $\M$ for $\M(d)$. First we show that $\M'$ is non-trivial. 
Let us denote $\sigma=(1 \ 2 \ \dots \ d)$ and $\tau=(1 \ 2 \ \dots \ d-1)$. We have
\begin{align*}
[m_1,m_2]&=\sigma^{-1}(1, \dots, 1, m_1^{-1})\tau^{-1}(1, \dots, 1, m_2^{-1}, m_1)\sigma(1, \dots, 1, m_2, 1)\tau\\ 
& =(1, \dots, 1, m_1^{-1})^{\sigma}(1, \dots, 1, m_2^{-1}, m_1)^{\tau\sigma}(1, \dots, 1, m_2, 1)^{\tau^{\sigma}}[\sigma,\tau]\\ 
&=(m_1^{-1}, 1, \dots, 1)(m_1, m_2^{-1}, 1, \dots, 1)(1, \dots, 1, m_2)(1 \ 2 \ d).
\end{align*}
Hence, we obtain that 
\begin{equation}\label{commm1m2}
[m_1,m_2]=(1, m_2^{-1}, 1, \dots, 1, m_2)(1 \ 2 \ d).
\end{equation}
By Equation \eqref{commm1m2}, we have $[m_1, m_2] \notin \St_{\M}(1)$, thus $\M'$ is non-trivial.

Now, for $i=1,\dots, d-2$, and $j=i+1,\dots, d-1$, we have
\begin{equation}\label{commmimj}
[m_i^{d+1-i}, m_j]^{m_1^{d-1}} =(1, \dots, 1, [m_i, m_j]).
\end{equation}
Then in order to prove that $\{1\} \times\cdots\times \{1\} \times  \M'\leq \psi(\M' \cap \St_{\M}(1))$, it only remains to show that for any $i=1, \dots, d-1$, there exists $x(i) \in \M'\cap \St_{\M}(1)$ such that
$$
x(i)=(1, \dots, 1, [m_i,m_d]).
$$
To find such $x(i)$, we first observe that
$$
[(m_i^{d+1-i})^{m_1^{i-1}}, m_d]=(1, \overset{i}{\dots}, 1, [m_{i}, m_{i+1}], \dots, [m_i, m_{d-1}], [m_i, m_d]).
$$
In order to cancel all these commutators above except for the last component, we use Equation \eqref{commmimj}, and we observe that since $\M$ is level transitive, if we conjugate with a suitable power of $m_1$, we get $[m_i, m_{i+1}]^{-1}, \dots, [m_i, m_{d-1}]^{-1}$ in each component. 
For example, if $i=2$, we have 
$$
[(m_2^{d-1})^{m_1}, m_d]=(1, 1, [m_2, m_3], [m_2, m_4], \dots, [m_2, m_d]).
$$
By using the considerations above, we obtain that $x(2)$ must be of the form
\begin{align*}
x(2)&=[m_3, m_2^{d-1}]^{m_1^2}[m_4, m_2^{d-1}]^{m_1^{3}}\dots[m_{d-1}, m_2^{d-1}]^{m_1^{d-2}}[(m_2^{d-1})^{m_1}, m_d]\\
&=(1, \dots, 1, [m_2, m_d]). 
\end{align*}
This concludes the proof.
\end{proof}

To prove last part of the main theorem (that $\M(d)$ is non-contracting), we need some preliminary tools. Namely, we show some results regarding the order of elements of $\M(d)$. We will handle the case $d=2$, and $d>2$ separately. More precisely, we first prove that $\M(2)$ is torsion-free, and then, for $d >2$, we show that the groups $\M(d)$ are neither torsion-free nor torsion, contrary to the case $d=2$.

The following Remark \ref{prodcomm} and Lemma \ref{zz} are key steps to prove that $\M(2)$ is torsion-free. We write $\M$ for $\M(2)$.

\begin{rem}\label{prodcomm}
Let $h\in \M'$ with $h=(h_1,h_2)$. Then $h_1h_2 \in \M'$.
\end{rem}
\begin{proof}
Consider the following map $\rho$:
\begin{alignat*}{3}
\rho: \St_{\M}(1) &\to& \M \ &\to& \ \M/\M'\\
(h_1, h_2) &\mapsto& \ h_1h_2 &\mapsto& \overline{h_1h_2}.
\end{alignat*}
Note that $\rho$ is a homomorphism of groups since $\M/\M'$ is abelian. As $\St_{\M}(1)/\Ker\rho$ is abelian, $\M' \leq \Ker\rho$. This concludes the proof.
\end{proof}


In the proof of next lemma, for a prime $p$ we denote with $\nu_p(m)$ the \textit{p-adic valuation} of $m$, that is the highest power of $p$ that divides $m$.

\begin{lem}\label{zz}
We have $\M'= (\M' \times \M')\langle [m_1,m_2] \rangle$. Furthermore
$$
\M/\M' \cong \langle m_1\M \rangle \times \langle m_2\M\rangle \cong \mathbb{Z} \times \mathbb{Z}.
$$
\end{lem}
\begin{proof}
Since $\M$ is weakly regular branch over $\M'$ by Proposition \ref{mweaklybranch}, and
$$
[m_1,m_2]=([m_1,m_2]m_2^{-1}m_1, m_1^{-1}m_2),
$$
we deduce that $(m_2^{-1}m_1, m_1^{-1}m_2)$ is an element of $\M'$. Furthermore, the elements $[m_1,m_2]^y$ where $y \in \{m_1, m_2, m_1^{-1}, m_2^{-1}\}$ are in $\langle [m_1,m_2] \rangle$ modulo $\M' \times \M'$ (note that $\M' \times \M'$ is normal in $\M$). More precisely, we have
\begin{align*}
 [m_1,m_2]^{m_1}&=({m_1}^{-2}m_2m_1,{m_1}^{-1}m_2^{-1}{m_1}^2)\\
 &=([{m_1}^2,{m_2}^{-1}]{m_2}{m_1}^{-1}, [m_1,m_2]{m_2}^{-1}m_1)\\ &\equiv ({m_1}^{-1}m_2, {m_2}^{-1}m_1) \mod \M' \times \M',  
\end{align*}
and similarly for the other commutators. Thus $\M'= (\M' \times \M')\langle [m_1,m_2] \rangle$, as required.

Now we want to show that if $m_1^im_2^j \in \M'$, then necessarily $i=j=0$. As $m_1^im_2^j \in \M' \leq \St_{\M}(1)$, then $i$ must be even. By way of contradiction, we choose the element $m_1^im_2^j \in \M'$ subject to the condition that $i$ is divisible by the least possible positive power of 2, say $2^a$, for some $a$. In other words, $\nu_2(i)=a$. Then if $m_1^rm_2^s \in \M'$, necessarily $2^a \mid r$. Note that it cannot happen that $r=0$ and $s\neq 0$ as $m_2$ is of infinite order. To prove the latter, we show that $m_1$ is of infinite order and as a consequence so is $m_2$, since $m_2=(m_1, m_2)$. By way of contradiction suppose that, for some $k$, $m_1$ has order $n=2k$, as $m_1$ has order $2$ modulo the first level stabilizer. We have
\begin{equation*}
m_1^n = (m_1^{k}, m_1^{k}) =(1, 1),
\end{equation*}
which yields a contradiction as $k<n$.
Now, writing $i=2i_1$ for some $i_1$, we have
$$
m_1^im_2^j=({m_1}^{i_1+j},{m_1}^{i_1}{m_2}^j) \equiv [m_1^k, m_2^k] \equiv ({m_1}^{k}{m_2}^{-k}, {m_1}^{-k}{m_2}^k) \mod \M'\times \M'.
$$
This implies that ${m_1}^{i_1+j-k}{m_2}^{k} \in \M'$ and ${m_1}^{i_1+k}{m_2}^{j-k} \in \M'$. As $2^a \mid i_1+j-k$ and $2^a \mid i_1+k$, then $2^a$ divides also $j$. This is because $2^a \mid 2i_1+j=i+j$ and by hypothesis $2^a \mid i$.  Finally, we also have ${m_1}^{i_1+k}{m_2}^{j-k} \in \M'$, from which we get
$$
{m_1}^{i_1+k}{m_2}^{j-k}=\left(m_1^{\frac{i_1+k}{2}+j-k}, m_1^{\frac{i_1+k}{2}}m_2^{j-k}\right).
$$
By Remark \ref{prodcomm}, we have $m_1^{i_1+j}m_2^{j-k}\in \M'$ which implies that $2^a \mid i_1+j$. As $\nu_2(i_1)=a-1$ and $2^a \mid j$, then $\nu_2(i_1+j)=a-1$, a contradiction as $2^a \mid i_1+j$. This completes the proof.
\end{proof}

As a consequence, we prove the following.

\begin{pr}
The group $\M(2)$ is torsion-free.
\end{pr}
\begin{proof}
Suppose by way of contradiction that there exists an element of finite order in $\M$. Since $\M/\M' \cong \mathbb{Z} \times \mathbb{Z}$ by Lemma \ref{zz}, then this element must lie in $\M' \leq \St_{\M}(1)$. Suppose that among all elements of finite order, we take the element $g$ that lies in $\St_{\M}(n) \setminus \St_{\M}(n+1)$, with $n$ minimum with this property. Write $g=(g_1, g_2)$. As $g$ is of finite order, then also $g_1, g_2$ must be of finite order. By our assumption, $g_1, g_2$ must lie at least in $\St_{\M}(n)$. This implies that $g \in \St_{\M}(n+1)$, a contradiction.
\end{proof}

In the following we determine the order of some elements of $\M(d)$, for $d>2$.

\begin{pr}
Let $d >2$. Then the group $\M(d)$ is neither torsion-free nor torsion.
\end{pr}
\begin{proof}
For ease of notation we write $\M$ for $\M(d)$. We start by proving that the given generators of $\M$ are of infinite order. Consider $m_1$, and suppose by way of contradiction that its order is $n$. Then if $m_1^n=1$, we obtain that $m_1^n$ must lie in  $\St_{\M}(1)$. Also, its order must be a multiple of $d$, say $n=dk$ for some $k$, since $m_1$ has order $d$ modulo the first level stabilizer. Since $m_1=(1,\dots, 1, m_1)(1 \ 2 \ \dots \ d)$, we obtain
\begin{align*}
m_1^n = (m_1^{k}, \dots, m_1^{k}) =(1, \dots, 1).
\end{align*}
This yields a contradiction since $m_1^k=1$ and $k<n$. Similar arguments can be used for the generators $m_2, \dots, m_{d-1}$, and $m_d$ has infinite order because $m_d=(m_1, \dots, m_d)$. Finally, $\M$ is not torsion-free since it contains elements of finite order; for example $[m_1,m_2]$ has order 3. By Equation \eqref{commm1m2}, we have 
$$
[m_1,m_2]=(1, m_2^{-1}, 1, \dots, 1, m_2)(1 \ 2 \ d).
$$
Thus it follows readily that $[m_1, m_2]^3=1$, as desired.
\end{proof}

We conclude the paper by proving the remaining part of the main theorem. 

\begin{pr}
The group $\M$ is non-contracting.
\end{pr}
\begin{proof}
Suppose by way of contradiction that $\M$ is contracting with nucleus $\mathcal{N}$. Notice that the element $m_d^{m_1}$ stabilizes the vertex 1. As a consequence, by induction, $m_d^{m_1}$ fixes all the vertices of the path $v=1 \overset{n}{\dots} 1$ for all $n \geq 1$. Also, $(m_d^{m_1})_{v}=m_d^{m_1}$. Clearly, this implies that $m_d^{m_1}$ lies in $\mathcal{N}$. Consider now a power $k$ of $m_d^{m_1}$. Arguing as before, we obtain again that $(m_d^{m_1})^k$ fixes $v$ and its section at $v$ is $(m_d^{m_1})^k$. Thus, $(m_d^{m_1})^k \in \mathcal{N}$ for any $k \geq 1$. This concludes the proof since $m_d^{m_1}$ has infinite order.
\end{proof}

\vspace{0.6cm}
\bibliography{bib}

\end{document}